\documentclass[a4paper,reqno,12pt]{article}
\usepackage{t1enc} 
\usepackage[margin=1in]{geometry}
\usepackage{xcolor}
\usepackage{graphicx}
\usepackage{amsmath,amsthm,amssymb}
\makeatletter
\let\@fnsymbol\@arabic
\makeatother

\newcommand{\ed}{\stackrel{(d)}{=}}
\def\p{\mathbb{P}}
\def\e{\mathbb{E}}
\def\ind{\mbox{\rm 1\hspace{-0.04in}I}}

\theoremstyle{plain}
\newtheorem{theorem}{Theorem}
\newtheorem{lemma}{Lemma}

\newtheorem{proposition}{Proposition}
\theoremstyle{definition}
\newtheorem{definition}{Definition}

\newtheorem{remark}{Remark}

\numberwithin{equation}{section}

\title{On $\mathbb{R}^d$-valued multi-self-similar Markov processes}

\author{Lo\"i{}c Chaumont\footnote{LAREMA UMR CNRS 6093, Universit\'e d'Angers, 2, Bd Lavoisier 
Angers Cedex 01, 49045, France. Email: loic.chaumont@univ-angers.fr}  \and 
Salem Lamine\thanks{University of Monastir, Faculty of sciences, Monastir and LAREMA UMR CNRS 6093, Universit\'e d'Angers, 
2, Bd Lavoisier Angers Cedex 01, 49045, France. Email: salem.lamine@hotmail.fr}}

\date{\today}

\begin{document}

\maketitle

\begin{abstract} 
 An $\mathbb{R}^d$-valued Markov process $X^{(x)}_t=(X^{1,x_1}_t,\dots,X^{d,x_d}_t)$, 
$t\ge0,x\in\mathbb{R}^d$ is said to be multi-self-similar with index $(\alpha_1,\dots,\alpha_d)\in[0,\infty)^d$ if the identity 
in law 
\[(c_iX_t^{i,x_i/c_i};i=1,\dots,d)_{t\ge0}\ed(X_{ct}^{(x)})_{t\ge0}\,,\] where $c=\prod_{i=1}^dc_i^{\alpha_i}$,
is satisfied for all $c_1,\dots,c_d>0$ and all starting point $x$. Multi-self-similar Markov processes were 
introduced by Jacobsen and Yor \cite{jy} in the aim of extending the Lamperti transformation of positive 
self-similar Markov processes to $\mathbb{R}^d_+$-valued processes. This paper aims at giving a complete description
of all $\mathbb{R}^d$-valued multi-self-similar Markov processes. We show that their state space
is always a union of open orthants with 0 as the only absorbing state and that there is no finite entrance law at 0 for these 
processes. We give conditions for these processes to satisfy the Feller property. Then we show that a Lamperti-type representation 
is also valid for $\mathbb{R}^d$-valued multi-self-similar Markov processes. In particular, we obtain a one-to-one relationship 
between this set of processes and the set of Markov additive processes with values in $\{-1,1\}^d\times\mathbb{R}^d$.  We then 
apply this representation to study the almost sure asymptotic behavior of multi-self-similar Markov processes. 
\end{abstract}

\noindent {\it Keywords}: Multi-self-similarity, Markov additive process, L\'evy process, time change.\\

\noindent {\it AMS MSC 2010}: 60J45

%
%

\section{Introduction}

A Markov process $\{(X_{t})_{t\ge0},\p_x\}$ with state space $E\subset\mathbb{R}^d$ satisfies the multi-scaling property of index 
$\alpha=(\alpha_1,\dots,\alpha_d)\in[0,\infty)^d$ if for all $x=(x_1,\dots,x_d)\in E$ and $c=(c_1,\dots,c_d)\in(0,\infty)^d$, 
\begin{equation}\label{8499}
\{(X_{c^\alpha t})_{t\ge0},\p_{c\circ x}\}=\{(c\circ X_t)_{t\ge0},\p_{x}\}\,,
\end{equation}
where $c^\alpha:=\prod_{i=1}^dc_i^{\alpha_i}$ and $\circ$ denotes the Hadamard product, that is $c\circ x:=(c_1x_1,\dots,c_dx_d)$. 
These processes, called multi-self-similar Markov processes, were introduced by Jacobsen and Yor in \cite{jy} in order to extend the 
famous Lamperti representation to $(0,\infty)^d$-valued Markov processes. They proved that any $(0,\infty)^d$-valued multi-self-similar 
Markov process $\{(X^{(1)}_{t},\dots,X^{(d)}_{t})_{t\ge0},\p_x\}$ can be represented as
\begin{equation}\label{3797}
X^{(i)}_t=\exp \xi^{(i)}_{\tau_t}\,,
\end{equation}
where $(\xi^{(1)},\dots,\xi^{(d)})$ is a $d$-dimensional L\'evy process  issued from $(\log x_1,\dots,\log x_d)$ and 
$\tau_t=\inf\{s:\int_0^s\exp(\alpha_1\xi^{(1)}_u+\dots+\alpha_d\xi^{(d)}_u)\,du>t\}$. They also proved that conversely, for any L\'evy 
process $(\xi^{(1)},\dots,\xi^{(d)})$, the transformation (\ref{3797}) defines a $(0,\infty)^d$-valued multi-self-similar Markov process.\\

On the other hand, in the recent work \cite{acgz}, the authors showed that there is a way to extend the Lamperti representation 
to all standard self-similar Markov processes with values in $\mathbb{R}^d\setminus\{0\}$. This extension defines a bijection 
between the set of these processes and this of $S_{d-1}\times \mathbb{R}$-valued Markov additive processes, where 
$S_{d-1}$ is the sphere in dimension $d$. Actually this representation does not provide a complete description of
$\mathbb{R}^d$-valued self-similar Markov processes since it does not give any information on the existence of an entrance law at 0
or a recurrent extension after the first passage time at 0. Regarding these questions, only the real case has been investigated up to 
now, see \cite{ri}, \cite{LT1}, \cite{ppr} and the references therein.\\ 
      
We show in this paper that unlike for self-similar Markov processes, a complete description of $\mathbb{R}^d$-valued 
multi-self-similar Markov processes can be given through a Lamperti-type representation. The multi-scaling property generates quite 
specific properties of the process. In particular any element $x\in\mathbb{R}^d$ such that $x_1x_2\dots x_d=0$ is absorbing and 
the state space can always be reduced to a union of open orthants with 0 as the only absorbing state. It implies that there is no 
continuous multi-self-similar Markov process whose state space covers the whole set $\mathbb{R}^d$. Moreover there is no finite 
multi-self-similar entrance law for these processes. We will then prove that, provided the process has infinite lifetime, the Feller 
property is satisfied on the whole state space, see Theorem~\ref{3457}. These features will be proved in Section \ref{3427} and will 
be used in Section \ref{6200} to show that (\ref{3797}) can be extended for $\mathbb{R}^d$-valued processes, so that
\begin{equation}\label{5727}
X^{(i)}_t=J^{(i)}_{\tau_t}\exp \xi^{(i)}_{\tau_t}\,,
\end{equation}
where $(J^{(i)},\xi^{(i)})_{1\le i\le d}$ is a Markov additive process with values in $\{-1,1\}^d\times\mathbb{R}^d$, and  
$\tau_t=\inf\{s:\int_0^s\exp(\alpha_1\xi^{(1)}_u+\dots+\alpha_d\xi^{(d)}_u)\,du>t\}$, see Theorem \ref{2461}.\\ 

Markov additive processes can be considered as generalizations of $d$-dimensional L\'evy processes. Roughly speaking, 
$(\xi^{(i)})_{1\le i\le d}$  behaves like a new L\'evy process in between each pair of jumps of the continuous time Markov chain 
$(J^{(i)})_{1\le i\le d}$. Together with the representation (\ref{5727}) they provide an easy means to construct many concrete 
examples of multi-self-similar Markov processes. This quite simple structure will also be exploited for the description of their path 
properties. We will show in Subsection \ref{5732}, see Theorem \ref{2774}, that actually these processes cannot reach the set 
$\{x\in\mathbb{R}^d\setminus\{0\}:x_1x_2\dots x_d=0\}$ in a continuous way. Moreover, contrary to self-similar Markov 
processes, the finiteness of their lifetime depends on their scaling index. Then we will describe the behavior of multi-self-similar 
Markov processes on the left neighborhood of their lifetime. In particular, we will study the existence of a limit and give conditions
for this limit to be 0, when it exists.

\section{General properties of mssMp's}\label{3427}

\subsection{The Markovian framework.}\label{6322}
Let us first set some notation. For $a=(a_1,\dots,a_d)\in[0,\infty)^d$ and $b=(b_1,\dots,b_d)\in\mathbb{R}^d$, we set 
$a^b=a_1^{b_1}\dots a_d^{b_d}$. For $x\in\mathbb{R}^d$, we set $\mbox{sign}(x):=(\mbox{sign}(x_1),\dots,\mbox{sign}(x_d))$, 
where $\mbox{sign}(0)=0$, and for all $s\in\{-1,0,1\}^d$, we define the set,
\begin{equation}\label{3089}
Q_s=\{x\in\mathbb{R}^d:\mbox{sign}(x)=s\}.
\end{equation}
If $s\in\{-1,1\}^d$, then $Q_s$ is called an open orthant. In all the remainder of this work, unless explicitly stated, we will 
assume that $d\ge2$. Let us emphasize that most of our results would not apply for $d=1$. However in the latter case 
multi-self-similar Markov processes coincide with self-similar Markov processes which have already been extensively studied 
in the literature, see \cite{la}, \cite{cpr}, \cite{LT1}, \cite{ri} or \cite{ppr}, for instance.\\

In this subsection, we give a proper definition of multi-self-similar Markov processes and describe the general form of
their state space. Let $E$ be a subset  of $\mathbb{R}^d$ which is locally compact with a countable base and let 
$\mathcal{E}$ be its Borel $\sigma$-field. 
Let 
\[(\Omega,\mathcal{F},(\mathcal{F}_t)_{t\ge0},(X_t)_{t\ge0},(\p_x)_{x\in E})\] 
be a Markov process with values in $(E,\mathcal{E})$, where $(\Omega,\mathcal{F})$ is some measurable space, 
$(\p_x)_{x\in E}$ are probability measures on $(\Omega,\mathcal{F})$ such that $\p_x(X_0=x)=1$, for all $x\in E$
and $(\mathcal{F}_t)_{t\ge0}$ is some filtration to which $X=(X_t)_{t\ge0}$ is adapted and completed with respect to the 
measures $(\p_x)_{x\in E}$. Such a process will be denoted by $\{X,\p_x\}$ and will simply be referred to as an $E$-valued 
Markov process.\\

We will now consider $E$-valued Markov processes which satisfy the multi-scaling property (\ref{8499}). We stress the
fact that this property  supposes the condition: 
\[Q_{\tiny\mbox{sign}(x)}\subset E,\;\;\;\mbox{for all $x\in E$.}\]
Define $T=\inf\{t:X_t\neq X_0\}$. We say that $x\in E$ is a holding state if $\p_x(T>0)=1$ and that it is an absorbing state if 
$\p_x(T=\infty)=1$. 
 
\begin{proposition}\label{2683} 
Let  $\{X,\p_x\}$ be an $E$-valued Markov process. Assume moreover that $\{X,\p_x\}$ is right continuous and satisfies 
$(\ref{8499})$. Then,
\begin{itemize}
\item[$1.$] each $x\in E$ such that $x_1x_2\dots x_d=0$ is an absorbing state. 
\item[$2.$]  If $x\in E$ is a holding sate $($resp.~an absorbing state$)$, then all elements of the set $Q_{\tiny\mbox{\rm sign}(x)}$ 
are holding sates $($resp.~absorbing sates$)$.
\end{itemize}
\end{proposition}
\begin{proof} Let $x\in E$ such that $x_1x_2\dots x_d=0$. With no loss of generality we can assume that $x_1=0$. 
Let $c\in(0,\infty)^d$ such that $c_2=c_3=\dots=c_d=1$. 
Then the multi-scaling property (\ref{8499}) entails that the following identity in law holds for all $c_1>0$ and $t\ge0$,
\begin{equation}\label{3466}
\{(X^{(1)}_{c_1^\alpha t},X^{(2)}_{c_1^\alpha t},\dots,X^{(d)}_{c_1^\alpha t}),\p_{(0,x_2,\dots,x_d)}\}=
\{(X^{(1)}_{t},X^{(2)}_{t},\dots,X^{(d)}_{t}),\p_{(0,x_2,\dots,x_d)}\}\,.\end{equation}
But letting $c_1$ go to 0 and using the fact that $\{X,\p_x\}$ is right continuous at 0, we obtain that for all $t\ge0$,
\[\p_{(0,x_2,\dots,x_d)}((X^{(1)}_{t},X^{(2)}_{t},\dots,X^{(d)}_{t})=(0,x_2,\dots,x_d))=1\,.\]
Therefore the state $x$ must be an absorbing state.

The second assertion follows directly from the multi-scaling property (\ref{8499}).
\end{proof}

\noindent Let $\{X,\p_x\}$ be a Markov process satisfying the conditions of Proposition \ref{2683}. Since all $x\in E$ such that 
$x_1x_2\dots x_d=0$ are absorbing, we can send the process to 0 whenever it reaches the set  $\{x\in E:x_1x_2\dots x_d=0\}$. 
Moreover, from part 2 of Proposition \ref{2683}, if $x\in E$ is an absorbing state such that $x_i\neq0$, for all $i=1,\dots,d$, then
all states of the orthant $Q_{\tiny\mbox{\rm sign}(x)}$  are absorbing. Then we will remove absorbing orthants as well as the 
set $\{x\in E:x_1x_2\dots x_d=0\}\setminus\{0\}$ from the state space so that with no loss of generality we can claim that any right 
continuous Markov process satisfying $(\ref{8499})$ has a state space of the form:
\begin{equation}\label{2423}
E\cup\{0\}\;\;\;\mbox{where}\;\;\;E=\cup_{s\in S} Q_s\,,
\end{equation}
and where $S$ is some subset of $\{-1,1\}^d$. Moreover, 0 is the only absorbing state.\\ 

From now on $E$ will always be a set of the form given in (\ref{2423}). In order to define multi-self-similar Markov processes, we 
need further usual assumptions. In particular, we will consider the set $E_0:=E\cup\{0\}$ as the Alexandroff one-point compactification 
of $E$. This means that the open sets of $E_0$ are those of $E$ and all sets of the form $0\cup K^c$, where $K$ is a compact subset 
of $E$. The later sets form a neighborhood system for 0 and this particular state is called the point at infinity. 
Then $E_0$ endowed with this topology is a compact space.  In particular, for a sequence $x^{(n)}$ of $E_0$, $\lim_{n}x^{(n)}=0$ 
if and only if 
\begin{equation}\label{3697}
\lim_n\min\left(|x_i^{(n)}|,|x_i^{(n)}|^{-1},\,i=1,\dots,d\right)=0\,.
\end{equation}
We denote by $\mathcal{E}_0$, the Borel $\sigma$-fields of $E_0$. 
Let us set  $\zeta:=\inf\{t:X_t=0\}$. The random time $\zeta$ is called the lifetime of $\{X,\p_x\}$, and the latter process is said to 
be absorbed at 0.\\ 

In all the remainder this paper, we will be dealing with Hunt processes which we recall the definition from Section I.9 of \cite{bg} and 
Section A.2 of \cite{fom}. An $E$-valued Markov process $\{X,\p_x\}$ absorbed at 0 is a Hunt process if:
\begin{itemize}
\item[$(i)$] it is a strong Markov process,
\item[$(ii)$] its paths are right continuous on $[0,\infty)$ and have left limits on $(0,\infty)$,
\item[$(iii)$] it has quasi-left continuous paths on $(0,\infty)$.
\end{itemize}
Such a process will be called an $E$-valued Hunt process absorbed at 0. 

\begin{definition} A multi-self-similar Markov process $($mssMp$)$ with index $\alpha\in[0,\infty)^d$ is an $E$-valued Hunt 
process absorbed at $0$, which satisfies the multi-scaling property $(\ref{8499})$. The state space of mssMp's is always of the 
form given in $(\ref{2423})$ and $0$ is the only absorbing state. In the sequel, such a process will simply be referred to as an 
$E$-valued mssMp with index $\alpha\in[0,\infty)^d$.
\end{definition}

\noindent Some examples of mssMp's with state space $E=(0,\infty)^d$ are given in \cite{jy}. For more general sets $E$ of the 
form (\ref{2423}), let us mention the following simple examples.\\ 

\noindent 1) Let $(J_t^{(1)},\dots,J_t^{(d)})_{t\ge0}$ be any continuous time Markov chain with values in some subset $S$ of 
$\{-1,1\}^d$ and starting at $(1,1,\dots,1)\in S$. Let $\alpha\in[0,\infty)^d$ and set $\bar{\alpha}=\alpha_1+\dots+\alpha_d$. 
For each $x\in E$, let $\p_x$ be the probability measure which assigns to the process $X$ the law of 
\begin{equation}\label{6432}
(x_1J_{t/x^{\bar{\alpha}}}^{(1)},\dots,x_dJ_{t/x^{\bar{\alpha}}}^{(d)})\,,\;\;\;t\ge0.
\end{equation}
Then we readily check that $\{X,\p_x\}$ is an $E$-valued mssMp with index $\alpha$.\\

\noindent 2) A slightly more sophisticated example is given by the law of 
\begin{equation}\label{6702}
\left(x_1(1+\bar{\alpha}tx^{-\alpha})^{1/\bar{\alpha}}J^{(1)}_{\ln(1+\bar{\alpha}tx^{-\alpha})^{1/\bar{\alpha}}},\dots,
x_d(1+\bar{\alpha}tx^{-\alpha})^{1/\bar{\alpha}}J^{(d)}_{\ln(1+\bar{\alpha}tx^{-\alpha})^{1/\bar{\alpha}}}\right)\,,\;\;\;t\ge0,
\end{equation} 
where we assume that $\bar{\alpha}>0$. Again the process $\{X,\p_x\}$ is an $E$-valued mssMp with index $\alpha$. \\

\noindent 3) We called our third example "the jumping spider". Let $\{X,\p_x\}$ be a process which starts at time $t=0$, at some 
point $x\in E$ and runs along the axis $(0,x)$ as a reflected Brownian motion. Then at some time before 
hitting 0, the process jumps out to some other state of $y\in E$ and runs in the same way along the 
axis $(0,y)$, and so on. More specifically let $(R_t)_{t\ge0}$ be a reflected Brownian motion independent of the process 
$(J_t^{(1)},\dots,J_t^{(d)})_{t\ge0}$ defined above and such that $R_0=1$, a.s. Let $\alpha$ be such that $\bar{\alpha}=2$ 
and define
\[X^{(i)}_t=\left\{\begin{array}{ll}
x_iJ_{\int_0^{t/x^\alpha}R^{-2}_s\,ds}^{(i)}R_{t/x^\alpha}\,,&0\le t<x^\alpha\zeta,\\
0\,,&t\ge x^\alpha\zeta,
\end{array}\right.\]
where $\zeta=\inf\{s:R_s=0\}$. Then we can check that $\{X,\p_x\}$ is a mssMp with index $\alpha$, which is absorbed at 0,
at time $\zeta$.\\

\noindent Note that in these three examples, the decomposition (\ref{2423}) of the space $E$ is determined by the state space $S$
of the continuous time Markov chain $(J_t^{(1)},\dots,J_t^{(d)})_{t\ge0}$. An extension of the above constructions of mssMp's will 
be given in Section \ref{6200} through a Lamperti type transformation, see Theorem \ref{2461}.

\begin{remark}
It is straightforward from $(\ref{8499})$ that mssMp's of index $\alpha$ are also self-similar Markov processes with index 
$\alpha_1+\dots+\alpha_d$. However multi-self-similarity imparts to Markov processes much richer properties which 
could not be derived from the study of self-similar Markov processes given in \cite{acgz}. See for instance Theorem $\ref{3457}$
and Proposition $\ref{3673}$ below.
\end{remark}

\subsection{On the Feller property of mssMp's.}\label{3656}

Proposition \ref{2683} means in particular that given any mssMp $\{X,\p_x\}$, there does not exist any Markov process starting 
from 0 and with the same semigroup as $\{X,\p_x\}$. We will see in the next proposition that actually there is no finite 
multi-self-similar entrance law.\\

Let us now fix some definitions. In what follows, $\{X,\p_x\}$ will always be an $E$-valued mssMp with index $\alpha\in[0,\infty)^d$. 
We will denote by $(P_t)$ the transition function of $\{X,\p_x\}$. An entrance law for $\{X,\p_x\}$ is a family of non zero measures 
$\{\eta_t,t>0\}$ on $E$ satisfying the identity $\eta_sP_t=\eta_{t+s}$, that is for all nonnegative Borel function $f$ defined 
on $E$ and all $s,t>0$,
\[\int_E \e_x(f(X_t),t<\zeta)\,\eta_s(dx)=\int_E f(x)\,\eta_{t+s}(dx)\,.\]
We say that $\{\eta_t,t>0\}$  is a multi-self-similar entrance law if moreover there is a multi-index $\gamma\in[0,\infty)^d$ such that 
for all $c\in(0,\infty)^d$, 
\begin{equation}\label{7995}
\eta_s=c^{-\gamma}\eta_{sc^{-\alpha}}H_c\,,
\end{equation}
where $H_c$ denotes the dilation operator $H_cf(x)=f(c\circ x)$. This definition is the natural extension of self-similar entrance 
laws introduced in the framework of positive self-similar Markov processes in \cite{ri}, see (EL-i) and (EL-ii). 
In this paper the existence of self-similar entrance laws has been fully studied. Then in part 4.2 of \cite{ri} the author
extended his study to the case of mssMp's with values in $E=(0,\infty)^d$ whose radial part tends to infinity. In this particular 
case, the definition of multi-self-similarity of the entrance law corresponds to (\ref{7995}) for $\gamma=0$. An expression for 
the corresponding multi-self-similar entrance law can be found in \cite{ri}. In the next proposition we complete this result by showing 
that there cannot exist any finite such entrance law in general. 

\begin{proposition}\label{2383} 
MssMp's do not admit finite multi-self-similar entrance laws.
\end{proposition}
\begin{proof} Let $\{\eta_t,t>0\}$ be a multi-self-similar entrance law for $\{X,\p_x\}$ and let $f$ be any positive Borel
function defined on $E$. Let $t>0$ and $a>0$, then from (\ref{7995}) applied to $c=(1,a,1,\dots,1)$, we obtain 
\begin{equation}\label{5623}
\int_Ef(x_1,\dots,x_d)a^{\gamma_2}\eta_t(dx)=\int_Ef(x_1,ax_2,\dots,x_d)\eta_{t/a^{\alpha_2}}(dx)\,.
\end{equation}
Let $\pi_i:x\mapsto x_i$ be the projection on the $i$-th coordinate and set $U_i=\pi_i(E)$. 
Denote by $\eta'_t=\eta_t\circ\pi_1^{-1}$ the image of $\eta_t$ by $\pi_1$ on $U_1$. It follows from the above identity that 
\begin{equation}\label{3488}
a^{\gamma_2}\eta'_{t}=\eta'_{t/a^{\alpha_2}}\,.
\end{equation} 
If $\alpha_2=0$ and $\gamma_2>0$, then $\eta'_t$ is necessarily the infinite measure. Denote by $\eta''_t=\eta_t\circ\pi_2^{-1}$ 
the image of $\eta_t$ by $\pi_2$ on $U_2$. If $\alpha_2=\gamma_2=0$,
then from (\ref{5623}), $\eta''_t$ satisfies 
\[\int_{U_2}g(x_2)\eta''_t(dx)=\int_{U_2}g(ax_2)\eta''_{t}(dx)\,,\]
for all positive Borel function $g$ defined on $U_2$. Recalling that either $U_2$ or $-U_2$ is the multiplicative group
$\mathbb{R}\setminus\{0\}$ or $(0,\infty)$, the latter identity shows that  if $\eta''_t$ is finite on all compact sets, 
then $\eta''_t$ corresponds to the Haar measure on $U_2$, that is $\eta''_t(dx)=(\mbox{cst})\cdot |x|^{-1}dx$, which has infinite 
mass.\\

Now suppose that $\alpha_2>0$. Then applying (\ref{7995}) again with $c=(a,1,\dots,1)$, we obtain
\[\int_Ef(x_1,\dots,x_d)a^{\gamma_1}\eta_t(dx)=\int_Ef(ax_1,x_2,\dots,x_d)\eta_{t/a^{\alpha_1}}(dx)\,,\]
that is for all positive Borel function $g$ defined on $U_1$,
\begin{equation}\label{3568}
\int_{U_1}g(x_1)a^{\gamma_1}\eta'_t(dx)=\int_{U_1}g(ax_1)\eta'_{t/a^{\alpha_1}}(dx)\,.
\end{equation} 
Then replacing $a$ by $a^{\alpha_1/\alpha_2}$ in (\ref{3488}) gives 
\begin{equation}\label{3412}
a^{\gamma_2\alpha_1/\alpha_2}\eta'_{t}=\eta'_{t/a^{\alpha_1}}\,,
\end{equation} 
so that from (\ref{3568}) and (\ref{3412}) with $\kappa=\gamma_1-\gamma_2\alpha_1/\alpha_2$,
\[\int_{U_1}g(ax_1)\eta'_t(dx)=\int_{U_1}g(x_1)a^{\kappa}\eta'_{t}(dx)\,.\]
But again, either $U_1$ or $-U_1$ is the multiplicative group $\mathbb{R}\setminus\{0\}$ or $(0,\infty)$, hence the 
latter identity shows that if $\eta'_t$ is finite on all compact sets, then $\eta'(dx)=(\mbox{cst})\cdot |x|^{-(\kappa+1)}dx$, which 
has infinite mass.\\ 
\end{proof}

\noindent Knowing the form of the state space of mssMp's and their entrance boundaries, we can now investigate 
their Feller property. Actually there is no universally agreed definition of the Feller property. This varies depending on which 
space the transition function $(P_t)$ of $\{X,\p_x\}$ should be defined. Let $\mathcal{C}_b(E_0)$ (resp.~$\mathcal{C}_b(E)$) 
be the space of continuous and bounded functions on $E_0$ (resp.~$E$). In our case, the most natural definition should 
require the following two conditions:
\begin{itemize}
\item[$(a)$] For all $f\in\mathcal{C}_b(E_0)$ and $t\ge0$, $P_tf\in\mathcal{C}_b(E_0)$.
\item[$(b)$] For all $f\in\mathcal{C}_b(E_0)$, $\lim_{t\rightarrow0}P_tf=f$, uniformly.
\end{itemize}
If the transition function of $\{X,\p_x\}$ satisfies $(a)$ and $(b)$, we say that $\{X,\p_x\}$ is a Feller process on $E_0$. 
As will be seen later on, this property is actually very strong and rarely satisfied by mssMp's. We will actually focus our attention 
on mssMp's with an infinite lifetime. A mssMp $\{X,\p_x\}$ is said to have an infinite lifetime if  $\p_x(\zeta=\infty)=1$, for all $x\in E$. 
When the lifetime is infinite, the restriction of the transition function $(P_t)$ of $\{X,\p_x\}$ to the space 
$E$ is still Markovian. In this case, we say that $\{X,\p_x\}$ has the Feller property on $E$  if it satisfies: 
\begin{itemize}
\item[$(a')$] For all $f\in\mathcal{C}_b(E)$ and $t\ge0$, $P_tf\in\mathcal{C}_b(E)$.
\item[$(b')$] For all $f\in\mathcal{C}_b(E)$, $\lim_{t\rightarrow0}P_tf=f$, uniformly on compact subsets of $E$. 
\end{itemize}
Examples of mssMp's whose lifetime is infinite  can be found in \cite{jy}. 

\begin{theorem}\label{3457}  
Let $\{X,\p_x\}$ be a mssMp with an infinite lifetime. Then, $\{X,\p_x\}$ has the Feller property 
on $E$. 
\end{theorem}
\begin{proof} Let us define the scaling operator $S_c$ by 
\[S_c(X)=(c\circ X_{t/{c}^\alpha})_{t\ge0}\,,\]
where $c\in(0,\infty)^d$. Then let $x\in E$ and $x^{(n)}=(x^{(n)}_1,\dots,x^{(n)}_d)$ be any 
sequence of $E$ which converges towards $x\in E$. Recall the form (\ref{2423}) of $E$ and set 
$Q^{(i)}=Q_{\mbox{\tiny sign$(x_i)$}}$ so that $x_i\in Q^{(i)}$. 
Since the $Q^{(i)}$'s are open sets, there is $n_0$ such that for all $n\ge n_0$ and all $i=1,\dots,d$, $x^{(n)}_i\in Q^{(i)}$ and 
$x^{(n)}_i$ has the same sign as $x_i$. Now for $n\ge n_0$, define $c^{(n)}_i=x_{i}^{(n)}/x_i$ and note that 
$c^{(n)}=(c^{(n)}_1,\dots,c^{(n)}_d)\in (0,\infty)^d$. Let $f\in\mathcal{C}_b(E)$, then from (\ref{8499}), for all $t\ge0$,        
\begin{equation}\label{1955}
\e_{x^{(n)}}(f(X_t))=\e_x(f(S_{c^{(n)}}(X)_t))\,.
\end{equation}
Since $c^{(n)}$ tends to 1 and $\{X,\p_x\}$ is almost surely continuous at time $t$, see \cite{bg}, it follows from (\ref{1955}) and
dominated convergence that
\[\lim_{n\rightarrow\infty}\e_{x^{(n)}}(f(X_t))=\e_x(f(X_t))\,.\]
This proves $(a')$.\\

Let us now prove $(b')$. First observe that for all $x\in E$, 
\begin{equation}\label{3453}
\{X,\p_x\}=\{S_{c(x)}(X),\p_{\mbox{\tiny sign}(x)}\}\,,
\end{equation}
where $c_i(x)=|x_i|$.  Let $K$ be some compact subset of $E$. For $s\in\{-1,1\}^d$, define the compact subsets 
$K_s=Q_s\cap K$.
From (\ref{3453}) and the right continuity of $\{X,\p_x\}$ at 0, we have for all $y\in K_s$, 
\[\lim_{t\rightarrow0}|f(S_{c(y)}(X)_t)-f(y)|=0\,,\;\;\mbox{$\p_s$-almost surely,}\]
so that from the uniform continuity of $f$ on $K_s$, 
\[\lim_{t\rightarrow0}\sup_{y\in K_s}|f(S_{c(y)}(X)_t)-f(y)|=0\,,\;\;\mbox{$\p_s$-almost surely.}\]
Then the result follows from the inequality for all $x\in K$,
\[|P_tf(x)-f(x)|\le\max_{s\in\{-1,1\}^d}\e_s(\sup_{y\in K_s}|f(S_{c(y)}(X)_t)-f(y)|)\,,\]
the boundeness of $f$ and dominated convergence.\\ 
\end{proof}

\begin{remark}\label{0046} Let us emphasize that Theorem $\ref{3457}$ highlights a great difference between self-similarity 
and multi-self-similarity. Indeed, it is no true that all $d$-dimensional self-similar Markov processes with infinite lifetime satisfy 
the Feller property on $E$. As can been seen in the previous proof, self-similarity on all axis allows us to consider the limit of 
$P_tf(x_n)$ for any sequence $(x_n)$ converging to $x$ in $E$, whereas this convergence would only hold for sequences 
of the type $x_n=c_nx$, where $c_n>0\rightarrow1$ for self-similar Markov processes.\\ 
\end{remark}

\noindent We will say that a mssMp $\{X,\p_x\}$ is symmetric if it satisfies the two following conditions:
\begin{itemize}
\item[$(i)$] $E=s\circ E$ for all $s\in\{-1,1\}^d$ such that $(s\circ E)\cap E\neq\emptyset$.
\item[$(ii)$] $\p_x(X_t\in A)=\p_{s\circ x}(X_t\in s\circ A)$, for all $t\ge0$, $A\in\mathcal{E}$ and $s$ satisfying $(i)$.
\end{itemize}
Note that if $E$ satisfies condition $(i)$, then conditions (\ref{8499}) and $(ii)$ are equivalent to
\begin{equation}\label{8099}
\{(X_{|c|^\alpha t})_{t\ge0},\p_{c\circ x}\}=\{(c\circ X_t)_{t\ge0},\p_{x}\}\,,
\end{equation}
where $|c|^\alpha:=|c_1|^{\alpha_1}\dots |c_d|^{\alpha_d}$, for all $x\in E$ and $c\in(\mathbb{R}\setminus\{0\})^d$ 
such that $c\circ x\in E$. Note also that if $E$ consists in a single orthant, 
that is $E=Q_s$, for some $s\in\{-1,1\}^d$, then $\{X,\p_x\}$ is symmetric according to this definition. Moreover, we easily 
construct symmetric mssMp's with $E$ as the union of at least two orthants from the examples given in 
Subsection \ref{6322}.\\

We will see in the next proposition that when $\{X,\p_x\}$ is symmetric, the lifetime is either a.s. finite or a.s. infinite, 
independently of the starting state. Moreover, either the process hits 0 continuously, a.s.~or by a jump, a.s. In the next proposition, 
we will use the notation, 
\[\lim_{t\uparrow\zeta} X_t=X_{\zeta-}\]
when this limit exists. Note that when $\zeta$ is finite, the existence of $X_{\zeta-}$ is guaranteed by the fact that $\{X,\p_x\}$ is a 
Hunt process. Moreover, the fact that $X_{\zeta-}=0$ is to be understood in the topology of $E_0$. It means that, 
\[\lim_{t\uparrow\zeta}\min(|X_t^{(i)}|,|X_t^{(i)}|^{-1},\,i=1,\dots,d)=0,\;\;\;\mbox{a.s.},\]
see (\ref{3697}).

\begin{proposition}\label{3673} Assume that $\{X,\p_x\}$ is a symmetric mssMp. Then, 
\begin{itemize}
\item[$(i)$] either $\p_x(\zeta=\infty)=1$, for all $x\in E$, or $\p_x(\zeta<\infty)=1$, for all $x\in E$.
\item[$(ii)$] Assume that $\p_x(\zeta<\infty)=1$, for all $x\in E$. Then either $\p_x(X_{\zeta-}=0)=1$,
for all $x\in E$ or $\p_x(X_{\zeta-}\neq0)=1$, for all $x\in E$.
\end{itemize}
\end{proposition}
\begin{proof} Let us note that from our assumptions, for all $x,y\in E$, there is $c\in(\mathbb{R}\setminus\{0\})^d$ 
such that $y=c\circ x$ and $\{(X_{|c|^\alpha t})_{t\ge0},\p_{y}\}=\{(c\circ X_t)_{t\ge0},\p_{x}\}$, which yields the identity 
in law, 
\begin{equation}\label{5699}
\{|c|^\alpha\zeta,\p_x\}=\{\zeta,\p_y\},
\end{equation}
since $\zeta$ is the first passage time at 0 by $X$, i.e.~$\zeta=\inf\{t:X_t=0\}$. This observation allows us to extend the case of 
positive self-similar Markov processes which is treated in \cite{la} and from which our proof is inspired.\\

Let $F=\{\zeta<\infty\}$.  Then from (\ref{5699}), $\p_x(F)$ does not depend on $x\in E$. Let us set $\p_x(F)=p$. 
From the Markov property, one has for all $t>0$, 
\begin{eqnarray*}
\p_x(t<\zeta<\infty)&=&\e_x(\ind_{\{t<\zeta\}}\p_x(\mbox{$\exists
s\in(t,\infty)$, $X_s=0$}\,|\,\mathcal{F}_t))\\
&=&\e_x(\ind_{\{t<\zeta\}}\p_{X_t}(\zeta<\infty))=p\p_x(t<\zeta)\,,
\end{eqnarray*}
which leads to
\[p=\p_x(\zeta\le t)+\p_x(t<\zeta<\infty)=\p_x(\zeta\le
t)+p\p_x(t<\zeta)\,,\] so that $(1-p)\p_x(\zeta\le t)=0$. We conclude that either $p=1$, or $\p_x(\zeta\le t)=0$ for all 
$t\ge0$ and all $x\in E$, that is $\p_x(\zeta=+\infty)=1$ for all $x\in E$.\\

Let us now prove $(ii)$. Set $G=\{X_{\zeta-}=0\}$. Again, from our assumptions, for all $x,y\in E$, there is 
$c\in(\mathbb{R}\setminus\{0\})^d$ such that $y=c\circ x$ and
\[\{c\circ X_{\zeta-},\p_x\}=\{X_{\zeta-},\p_y\},\]
so that $\p_x(G)$ does not depend on $x\in E$. Set $q=\p_x(G)$ and let $K$ be any compact subset of $E$. Set 
$T=\inf\{t:X_t\in K^c\}$, where $K^c$ denotes the complementary set of $K$ in $E$. Since $G\subset\{T<\zeta\}$ and 
$G\circ\theta_{T}=G$, it follows from the strong Markov property that for all $x\in E$,
\begin{eqnarray*}
q=\p_x(G)&=&\p_x(G,T<\zeta)=\e_x\left(\ind_{\{T<\zeta\}}\p_x(G\circ\theta_{T}\,|\,\mathcal{F}_{T})\right)\\
&=&\p_x\left(\ind_{\{T<\zeta\}}\p_{X_{T}}(G)\right)=q\p_x(T<\zeta)\,.
\end{eqnarray*}
If $q\neq0$, then $\p_x(T<\zeta)=1$, for all $x\in E$.  Since this is true for all compact subsets of $E$, it follows that $\{X,\p_x\}$ 
doest not reach 0 by a jump and hence $q=1$.
\end{proof}
\noindent The Lamperti-type representation established  in Subsection \ref{2654} will allow us to give many other examples
of mssMp's with infinite lifetime, see Theorem \ref{2774}.\\  

Note that in order to have the Feller property on $E_0$, the process $\{X,\p_x\}$ should also satisfy
\begin{equation}\label{6582}
\lim_{x\rightarrow0}\e_x(f(X_t))=f(0)\,,
\end{equation}
for all $t\ge0$ and $f\in\mathcal{C}_b(E_0)$, where again this limit is to be understood in the topology of $E_0$, 
see (\ref{3697}). If $x$ tends to 0 (in $E_0$) in such a way that $|x_i|>a$ for some $a>0$ and all $i=1,\dots,d$, then (\ref{6582}) 
holds. Indeed, from the multiscaling property,
\[\e_{\mbox{\tiny sgn}(x)}(f(|x|\circ X_{t/|x|^{\alpha}})).\]
In this case, $t/|x|^{\alpha}$ tends to 0 as $x$ tends to 0 in $E_0$, so that from the right continuity of $X$ at 0,  
$\lim_{x\rightarrow0}X_{t/|x|^{\alpha}}=\mbox{ sgn}(x)$, $\p_{\mbox{\tiny sgn}(x)}$-a.s.~and hence 
$\lim_{x\rightarrow0}|x|\circ X_{t/|x|^{\alpha}}=0$, $\p_{\mbox{\tiny sgn}(x)}$-a.s. Then (\ref{6582}) follows from the fact that 
$f\in\mathcal{C}_b(E_0)$ and dominated convergence. However, it seems that (\ref{6582}) may fail when $\liminf |x_i|=0$ for some 
coordinates $x_i$ of $x$ since in this case, we can have $\liminf_{x\rightarrow0}|x|^\alpha=0$ and 
$\limsup_{x\rightarrow+\infty}|x|^\alpha=+\infty$.

\subsection{Multiplicative agglomeration property of mssMp's.}\label{5630}

We will prove in this subsection that  symmetric mssMp's enjoy the multiplicative agglomeration property, 
namely the process obtained by multiplying some of its coordinates is still a mssMp (with lower dimension). 
This property has been highlighted for $(0,\infty)^d$-valued mssMp's in \cite{jy} as a direct consequence of the Lamperti 
representation, see Corollary 5 therein. Although Lamperti representation will be generalized to all mssMp's later on in this paper, 
we prefer to study multiplicative agglomeration property of mssMp's in a more direct way by using Dynkin's criterion.\\

For $1\le d'\le d$ and a partition $I=\{I_1,\dots,I_{d'}\}$ of $\{1,2,\dots,d\}$, we define 
\[\Pi_I(x)=\left(\Pi_{i\in I_1} x_i,\dots,\Pi_{i\in I_{d'}} x_i\right)\,,\;\;\;x\in E\,.\]
Let us also define $E^{(I)}=\Pi_I(E)$. Then clearly $E^{(I)}$ is a subspace of $\mathbb{R}^{d'}$ which has the form described 
in (\ref{2423}). We denote by 
$\mathcal{E}^{(I)}$ the corresponding Borel $\sigma$-field and by $E^{(I)}_0$ the Alexandroff compactification of $E^{(I)}$
which is defined as for $E$, see Subsection \ref{6322}. Given any $E$-valued mssMp absorbed at 0, $\{X,\p_x\}$, we define the 
$E^{(I)}$-valued process absorbed at 0, $X^{(I)}$ by 
\[X^{(I)}=\Pi_I(X)\,,\]
and we set $\mathcal{F}^{(I)}=\sigma(X^{(I)}_t\in B,t\ge0, B\in\mathcal{E}^{(I)})$ and for all $t\ge0$, 
$\mathcal{F}^{(I)}_t=\mathcal{F}_t\cap\mathcal{F}^{(I)}$. 

\begin{proposition}\label{2742} Let $\{X,\p_x\}$ be a symmetric mssMp with index $\alpha\in[0,\infty)^d$ such that for all 
$i=1,\dots,d'$ and for all $j,k\in I_i$, $\alpha_j=\alpha_k$.  We set $\alpha'_i=\alpha_j$, for $j\in I_i$.
Then the process $X^{(I)}$ defined on the space $(\Omega,\mathcal{F}^{(I)},(\mathcal{F}_t^{(I)})_{t\ge0})$ is an 
$E^{(I)}$-valued mssMp absorbed at $0$ with index $\alpha'=(\alpha'_1,\dots,\alpha'_{d'})$. Moreover, the family of probability 
measures $\p^{(I)}_y$, $y\in E_0^{(I)}$associated with $X^{(I)}$ is given by 
\begin{equation}\label{3770}
\p^{(I)}_{y}(\Gamma)=\p_{x}(\Gamma),\;\;\;\Gamma\in\mathcal{F}^{(I)},
\end{equation}
for any $x\in\Pi_I^{-1}(y)$.
\end{proposition}

\begin{proof}  From Dynkin's criterion, see Theorem 10.23, p.325 in \cite{LT7}, in order to prove that $X^{(I)}$ is a Markov process 
defined on $(\Omega,\mathcal{F}^{(I)},(\mathcal{F}_t^{(I)})_{t\ge0})$, with respect to the family of probability measures $(\p_y^{(I)})$ 
given in (\ref{3770}), it suffices to prove that for all $x,x'\in E$ such that $\Pi_I(x)=\Pi_I(x')$, and for all $t\ge0$ and 
$B\in\mathcal{E}^{(I)}$,  
\begin{equation}\label{3615}
\p_x(X_t\in\Pi_I^{-1}(B))=\p_{x'}(X_t\in\Pi_I^{-1}(B)).
\end{equation}
Let $t\ge0$, $B\in\mathcal{E}^{(I)}$ and $x,x'\in E$ be such that $\Pi_I(x)=\Pi_I(x')$ and let us set
\[a_i=\frac{x_i}{x'_i}.\]
Since all coordinates of $\Pi_I(a)$, with $a=(a_1,\dots,a_d)$, are equal to~1 and from our assumption on $\alpha$, 
we have $|a|^{-\alpha}=1$. Hence, from (\ref{8099}),
\begin{eqnarray*}
\p_{x}(X_t\in\Pi_I^{-1}(B))&=&\p_{a\circ x'}(X_t\in\Pi_I^{-1}(B))\nonumber\\
&=&\p_{x'}(a\circ X_{|a|^{-\alpha}t}\in\Pi_I^{-1}(B))\nonumber\\
&=&\p_{x'}(X_t\in\Pi_I^{-1}(B)), 
\end{eqnarray*}
which is Dynkin's criterion (\ref{3615}).\\

It remains to check that the process $\{X^{(I)},\p_x^{(I)}\}$ satisfies the multi-self-similarity property of index $\alpha'$ defined
in the statement. This follows directly from the definition of $\{X^{(I)},\p_x^{(I)}\}$ and the multi-self-similarity property of $\{X,\p_x\}$. 
Indeed, recall that $d'$ is the dimension of $E^{(I)}$ and let $c'\in(0,\infty)^{d'}$, 
$y\in E^{(I)}$ and $x\in E$, $c\in(0,\infty)^d$ such that for all $j\in I_i$, $c_j=(c'_i)^{\mbox{\footnotesize card}(I_i)^{-1}}$ and 
$x\in\Pi_I^{-1}(y)$. Then,
\begin{eqnarray*}\{c'\circ X^{(I)},\p_{y}^{(I)}\}&=&\{\Pi_I(c\circ X),\p_x\}\\
&=&\{\Pi_I(X_{c^{\alpha}\cdot}),\p_{c\circ x}\}\\
&=&\{X^{(I)}_{(c')^{\alpha'}\cdot},\p_{c'\circ y}^{(I)}\},
\end{eqnarray*} 
which achieves the proof of the proposition. 
\end{proof}

Note that we recover Jacobsen and Yor's result from Proposition \ref{2742} since the process is always symmetric when 
$E=(0,\infty)^d$. Let us also mention that symmetry is not a necessary condition for the process $\{X^{(I)},\p_x^{(I)}\}$ to be a 
mssMp. Examples of non symmetric mssMp's which satisfy the multiplicative agglomeration property can be obtained from the 
Lamperti type representation presented in the next sections. We also emphasize the importance for the coordinates of the index 
$\alpha$ to be constant on each element of the partition $I$. The above proof shows that it necessary for the process $X^{(I)}$ 
to be Markovian. This fact is easier to see from the Lamperti type representation, see Theorem \ref{2461}.    

\section{Time changes in mssMp's.}\label{6200}

\subsection{Markov Additive Processes.}\label{2004}

We will now consider Markov processes with values in a state space of the form $S\times\mathbb{R}^d$, where $S$ is some 
topological set such that $S\times\mathbb{R}^d$ is locally compact with a countable base. As usual we define the Alexandroff 
compactification of $S\times\mathbb{R}^d$ by adding a point at infinity which we denote by $\delta$.

\begin{definition}\label{1594}
A Markov additive process (MAP) $\{(J,\xi),P_{y,z}\}$ is an $S\times\mathbb{R}^d$-valued Hunt process absorbed at some extra 
state $\delta$, such that for any $y\in S$, $z\in\mathbb{R}^d$, $s,t\ge0$, and for any positive measurable function $f$, defined on 
$S\times\mathbb{R}^d$, 
\begin{equation}\label{3473}
E_{y,z}(f(J_{t+s},\xi_{t+s}-\xi_t),t+s<\zeta_*\,|\,\mathcal{G}_t)=E_{J_t,0}(f(J_{s},\xi_{s}),s<\zeta_*)\ind_{\{t<\zeta_*\}}\,,
\end{equation}
where $\zeta_*=\inf\{t:(J_t,\xi_t)=\delta\}$ is the lifetime of $\{(J,\xi),P_{y,z}\}$ and $(\mathcal{G}_t)_{t\ge0}$ is some filtration 
to which $(J,\xi)$ is adapted and completed by the measures $P_{y,z}$, $(y,z)\in S\times\mathbb{R}^d$.
\end{definition}

In what follows, we will always consider the case where $S$ is a subset of $\{-1,1\}^d$. Then the space $S\times\mathbb{R}^d$ 
is always locally compact with a countable base. Moreover, while the structure of general MAP's can turn out to be quite complicated 
(see \cite{es} and \cite{ci} where these processes were first introduced) the case where $S$ is a finite set is rather intuitive and can 
be plainly described. Since in this case, the process $(J_t)_{t\ge0}$ is nothing but a possibly absorbed continuous time Markov chain, 
it is readily seen from (\ref{3473}) that in between two successive jump times of $J$, the process $\xi$ behaves like a L\'evy process. 
Let us state this result more formally.\\

It is straightforward from (\ref{3473}) that the law of $\{J,P_{y,z}\}$ does not depend on $z$. Moreover, since $S$ is finite, 
$\{J,P_{y,z}\}$ is an $S$-valued continuous time Markov chain with lifetime $\zeta_*$, which may be sent to some extra state 
$\delta'$ for $t\ge\zeta_*$. Let us set $n:=2^d=\mbox{card}(S)$, then the law of the MAP $\{(J,\xi),P_{y,z}\}$ is characterized by 
the intensity matrix $Q=(q_{ij})_{i,j\in S}$ of $J$, the laws of $n$ possibly killed $\mathbb{R}^d$-valued L\'evy processes 
$\tilde{\xi}^{(1)},\dots,\tilde{\xi}^{(n)}$, and the $\mathbb{R}^d$-valued random variables $\Delta_{ij}$, such that $\Delta_{ii}=0$ and 
where, for $i\neq j$, $\Delta_{ij}$ represents the size of the jump of $\xi$ when $J$ jumps from $i$ to $j$. More specifically, for 
$u\in\mathbb{C}^d$, define for $i,j\in S$ and $k=1,\dots,n$ when these expectations exist,
\[\e(e^{\langle u,\tilde{\xi}^{(k)}_1\rangle})=e^{\psi_k(u)}\;\;\;\;\mbox{and}\;\;\;\;G_{i,j}(u)=\e(\exp(\langle u,\Delta_{i,j}\rangle))\,.\]
Then a trivial extension of Proposition 2.2 in Section XI.2 of \cite{LT8} shows that the law of $\{(J,\xi),P_{y,z}\}$ is given by
\begin{equation}\label{7293}
E_{i,0}(e^{\langle u,\xi_t\rangle},J_t=j)=(e^{A(u)t})_{i,j}\,,\;\;\;i,j\in S\,,\;\;\;u\in\mathbb{C}^d\,,
\end{equation}
where $A(u)$ is the matrix,
\[A(u)=\mbox{diag}(\psi_1(u),\dots,\psi_{n}(u))+(q_{ij}G_{i,j}(u))_{i,j\in S}\,.\]
The matrix-valued mapping $u\mapsto A(u)$ will be called the characteristic exponent of the MAP $\{(J,\xi),P_{y,z}\}$. We also refer 
to Sections A.1 and A.2 of \cite{LT1} for more details. Throughout the remainder of this paper, the coordinates of a MAP with values 
in $S\times\mathbb{R}^d$ will be denoted by $(J,\xi)=(J^{(i)},\xi^{(i)})_{1\le i\le d}$.\\

Let $\{(J,\xi),P_{y,z}\}$ be an $S\times\mathbb{R}^d$-valued MAP with infinite lifetime, that is $P_{y,z}(\zeta_*=\infty)=1$, for all 
$y,z\in S\times\mathbb{R}^d$. Then it is readily seen that for each $k=1,\dots,d$, the process $(J,\xi^{(k)})$ is itself an 
$S\times\mathbb{R}$-valued MAP with infinite lifetime. Let $P^{(k)}_{y,z}$, $y,z\in S\times\mathbb{R}$ be the corresponding family 
of probability measures, i.e.~$\{(J,\xi^{(k)}),P^{(k)}_{y,z}\}$ is an $S\times\mathbb{R}$-valued MAP with infinite lifetime. Denote by 
$A_k$ the corresponding characteristic exponent, that is from (\ref{7293}),
\begin{equation}\label{5644}
A_k(u)=A(u\cdot e_k)\,,\;\;\;u\in\mathbb{C}\,,
\end{equation}
where $e_k$ is the $k$-th unit vector of $\mathbb{R}^d$. 
Fix $k=1,\dots,d$, assume that the Markov chain $(J_t)$ is irreducible and that there exists $u\in\mathbb{R}\setminus\{0\}$ such 
that $A_k(u)$, $i=1,\dots,d$ is well defined (i.e.~all entries of $A_k(u)$ exist and are finite). Then from Perron-Frobenius theory, 
the matrix $A_k(u)$ has a real simple eigenvalue $\chi_k(u)$ which is larger than the real part of all its other eigenvalues. Let 
$I=[u,0]$ if $u<0$ and $I=[0,u]$ if $u>0$. Then the function $u\mapsto\chi_k(u)$ is convex on $I$. Let us denote by $\chi_k'(0)$ 
the left (respectively, the right) derivative at 0 of $\chi_k$ if $u<0$ (respectively, if $u>0$). The following result can be found in 
Section XI.2 of \cite{LT8}.

\begin{proposition}\label{2690}
Assume that $J$ is irreducible. Let $k=1,\dots,d$ and assume that there exists $u\in\mathbb{R}\setminus\{0\}$ is such that 
$A_k(u)$ is well defined. Then the asymptotic behavior of $\xi^{(k)}$ does not depend on the initial state of 
$\{(J,\xi^{(k)}),P^{(k)}_{y,z}\}$ and is given by  
\[\lim_{t\rightarrow\infty}\frac{\xi^{(k)}_t}t=\chi'_k(0),\;\;\;\mbox{$P^{(k)}_{y,z}$-a.s. for all $y,z\in S\times\mathbb{R}$}.\]
In that case,  for all $y,z\in S\times\mathbb{R}$, $\lim_{t\rightarrow\infty}\xi^{(k)}_t=\infty$, $P^{(k)}_{y,z}$-a.s.~or 
$\lim_{t\rightarrow\infty}\xi^{(k)}_t=-\infty$, $P^{(k)}_{y,z}$-a.s.~or 
$\limsup_{t\rightarrow\infty}\xi^{(k)}_t=-\liminf_{t\rightarrow\infty}\xi^{(k)}_t=\infty$, $P^{(k)}_{y,z}$-a.s., according as 
$\chi'_k(0)>0$, $\chi'_k(0)<0$ or $\chi'_k(0)=0$, respectively. 
\end{proposition}
\noindent Note that more generally, if $M:\mathbb{R}^d\rightarrow\mathbb{R}^{d'}$ is a linear mapping, where $d'$ is any integer,
then the process $\{(J,M(\xi)),P_{y,z}\}$ is an $S\times\mathbb{R}^{d'}$-valued MAP. This property will be used in the next 
subsections with $M(x)=\langle\alpha,x\rangle$, for some $\alpha\in[0,\infty)^d$.\\ 

Examples of MAP's can easily be obtained by coupling any continuous time Markov chain on $S$ together with any $d$-dimensional 
L\'evy process and by killing the couple at some independent exponential time. More specifically, the transition probabilities of the 
process $\{(J,\xi),P_{y,z}\}$ have the following particular form:
\begin{equation}\label{3492}
\left\{\begin{array}{l}
P_{y,z}(J_t\in dy_1,\xi_t\in dz_1)=e^{-\lambda t}P_y^{J'}(J'_t\in dy_1)P_z^{\xi'}(\xi_t'\in dz_1)\,,\\
P_{y,z}((J_t,\xi_t)=\delta)=1-e^{-\lambda t}\,,
\end{array}\right.
\end{equation}
for all $t\ge0$, $(y,z),(y_1,z_1)\in S\times \mathbb{R}^d$, where $\lambda>0$ is some constant, $\{\xi',P_z^{\xi'}\}$ is any non killed 
$d$-dimensional L\'evy process and $\{J',P_y^{J'}\}$ is any continuous time Markov chain on $S$ with infinite lifetime. 
Then it is easy to check that this process $\{(J,\xi),P_{y,z}\}$ is an $S\times \mathbb{R}^d$-valued MAP which is absorbed at $\delta$,
in the sense of Definition \ref{1594}. The law of such a MAP is characterized by the fact that $\psi_k=\psi_l$ for all $k,l=1,\dots,n$ and 
$G_{i,j}(u)=1$ for all $i,j\in S$ in (\ref{7293}). 
Assume that the above process has infinite lifetime, that is $\lambda=0$. Then the condition of Proposition \ref{2690} is satisfied if and 
only if there exists $u\in\mathbb{R}$ such that $\psi_1(u)$ exists and is finite. In this example one may check that $\chi'(0)>0$, 
$<0$ or $=0$ according as $\psi_1'(0)>0$, $<0$ or $=0$. Of course this result is intuitively clear since $J$ and $\xi$ are independent. 

\subsection{The Lamperti representation for mssMp's.}\label{2654}

Recall that $S$ and $E$ are any sets such that 
\[S\subset\{-1,1\}^d\;\;\;\mbox{and}\;\;\;E=\cup_{s\in S}Q_s\,,\]
where $Q_s$ is defined in (\ref{3089}). Then let us define the one-to-one transformation $\varphi:S\times\mathbb{R}^d \rightarrow E$ 
and its inverse as follows:
\begin{eqnarray*}
\varphi(y,z)&=&(y_ie^{z_i})_{1\le i\le d}\,,\;\;\;(y,z)\in S\times\mathbb{R}^d\,,\\
\varphi^{-1}(x)&=&(\mbox{sgn}(x_i),\log(|x_i|))_{1\le i\le d}\,,\;\;\;x\in E\,.
\end{eqnarray*}
In the remainder of this work, for $\alpha\in[0,\infty)^d$, we will denote,  
\[\bar{\xi}=\langle\alpha,\xi\rangle\,,\]
where $\xi$ is the second coordinate of the $S\times\mathbb{R}^d$-valued MAP $\{(J,\xi),P_{y,z}\}$.\\

The next theorem extends Theorem 1 of \cite{jy}. It provides a one to one relationship between the set of $\mathbb{R}^d$-valued 
mssMp's and this of MAP's with values in $\{-1,1\}^d\times\mathbb{R}^d$.\\ 

\begin{theorem}\label{2461}
Let  $\alpha\in[0,\infty)^d$ and $\{(J,\xi),P_{y,z}\}$ be a MAP in $S\times\mathbb{R}^d$, with lifetime $\zeta_*$ and absorbing state 
$\delta$. Define the process $X$ by 
\[X_t=\left\{\begin{array}{lll}
\varphi(J_{\tau_t},\xi_{\tau_t})\,,&\mbox{if}&\mbox{$t<\int_0^{\zeta_*}e^{\bar{\xi}_s}\,ds$}\,,\\
0\,,&\mbox{if}&\mbox{$t\ge\int_0^{\zeta_*}\exp(\bar{\xi}_s)\,ds$}\,,
\end{array}\right.\]
where $\tau_t$ is the time change $\tau_t=\inf\{s:\int_0^se^{\bar{\xi}_u}\,du>t\}$, for $t<\int_0^{\zeta_*} e^{\bar{\xi}_s}\,ds$. Define 
the probability measures $\p_x:=P_{\varphi^{-1}(x)}$, for $x\in E$ and $\p_0:=P_\delta$. Then the process $\{X,\p_x\}$ is an 
$E$-valued mssMp, with index $\alpha$ and lifetime $\int_0^{\zeta_*}e^{\bar{\xi}_s}\,ds$.

Conversely, let $\{X,\p_x\}$ be an $E$-valued mssMp, with index $\alpha\in[0,\infty)^d$ and denote by $\zeta$ its lifetime. Define the 
process $(J,\xi)$ by 
\[(J_t,\xi_t)=\left\{\begin{array}{lll}
\varphi^{-1}(X_{A_t})\,,\;\;\mbox{if}&t<\int_0^{\zeta}\frac{ds}{|X_s^{(1)}|^{\alpha_1}\dots |X_s^{(d)}|^{\alpha_d}}\,,\\
\delta\,,\;\;\mbox{if}&t\ge \int_0^{\zeta}\frac{ds}{|X_s^{(1)}|^{\alpha_1}\dots |X_s^{(d)}|^{\alpha_d}}\,,
\end{array}\right.\]
where $\delta$ is some extra state, and $A_t$ is the time change 
$A_t=\inf\{s:\int_0^s\frac{du}{|X_u^{(1)}|^{\alpha_1}\dots |X_u^{(d)}|^{\alpha_d}}>t\}$, 
for $t<\int_0^{\zeta}\frac{ds}{|X_s^{(1)}|^{\alpha_1}\dots |X_s^{(d)}|^{\alpha_d}}$.  
Define the probability measures, $P_{y,z}:=\p_{\varphi(y,z)}$, for $(y,z)\in S\times\mathbb{R}^d$ and 
$P_\delta:=\p_0$. Then the process $\{(J,\xi),P_{y,z}\}$ is a MAP  in $S\times\mathbb{R}^d$, with lifetime 
$\int_0^{\zeta}\frac{ds}{|X_s^{(1)}|^{\alpha_1}\dots |X_s^{(d)}|^{\alpha_d}}$.
\end{theorem}
\begin{proof} Let us denote by $(\mathcal{G}_t)_{t\ge0}$ the filtration associated to the process $(J,\xi)$ and completed with respect 
to the measures $(P_{y,z})_{y,z\in S\times\mathbb{R}^d}$, see the beginning of Subsection \ref{6322}. Then define the process 
\[Y_t=\left\{\begin{array}{lll}
\varphi(J,\xi)_{t}\,,&\mbox{if}&\mbox{$t<\zeta_*$}\,,\\
0\,,&\mbox{if}&\mbox{$t\ge\zeta_*$}\,,
\end{array}\right.\]
and set $Y_t=(Y^{(1)}_t,\dots,Y^{(d)}_t)$ as usual. Recall from Subsection \ref{6322} the definition of $E_0$. Since $\varphi$ is a 
continuous one-to-one transformation from $S\times\mathbb{R}^d$ to $E$, we readily check that the process
\[(\Omega,\mathcal{F},(\mathcal{G}_t)_{t\ge0},(Y_t)_{t\ge0},(\p_x)_{x\in E_0}),\] 
where $\p_x$ is defined as in the statement, is an $E$-valued Hunt process absorbed at 0. Now let $\tau_t$ be as in the statement 
if $t<\int_0^{\zeta_*} e^{\bar{\xi}_s}\,ds$, and set $\tau_t=\infty$ and $Y_{\tau_t}=0$, if 
$t\ge\int_0^{\zeta_*} e^{\bar{\xi}_s}\,ds$. Since $e^{\bar{\xi}_s}=|Y_s^{(1)}|^{\alpha_1}|Y_s^{(2)}|^{\alpha_2}\dots|Y_s^{(d)}|^{\alpha_d}$,
 $(\tau_t)_{t\ge0}$ is the right continuous inverse of the continuous, additive functional 
$t\mapsto\int_0^{t\wedge\zeta_*}|Y_s^{(1)}|^{\alpha_1}|Y_s^{(2)}|^{\alpha_2}\dots|Y_s^{(d)}|^{\alpha_d}\,ds$ of $\{Y,\p_x\}$, which is 
strictly increasing on $(0,\zeta_*)$. It follows from Theorem A.2.12, p.406 in \cite{fom} that the time changed process
\[(\Omega,\mathcal{F},(\mathcal{G}_{\tau_t})_{t\ge0},(X_t)_{t\ge0},(\p_x)_{x\in E_0}),\]  is an $E$-valued Hunt 
process absorbed at 0. Moreover $\zeta:=\int_0^{\zeta_*} e^{\bar{\xi}_s}\,ds$ is the lifetime of $\{X,\p_x\}$.\\

Now let us show that $\{X,\p_x\}$ fulfills  the multi-scaling property. Let $c\in(0,\infty)^d$, then for 
$t<c^{-\alpha}\int_0^{\zeta_*} e^{\bar{\xi}_s}\,ds$,
\begin{equation}\label{2561}
\tau_{c^{\alpha}t}=\inf\{s:\int_0^se^{\langle\alpha,\xi^{(c)}_v\rangle}\,dv>t\}\,,
\end{equation}
where $\xi^{(c)}_t=(\xi^{(c,1)}_t,\dots,\xi^{(c,d)}_t)$ and $\xi^{(c,i)}_t=-\ln c_i+\xi_t^{(i)}$. It follows from Definition \ref{1594} of 
MAP's that
\begin{equation}\label{2562}
\{(J,\xi^{(c)}),P_{y,\ln c+z}\}=\{(J,\xi),P_{y,z}\},
\end{equation}
where $\ln c=(\ln c_1,\dots,\ln c_d)$. Let us set $\tau^{(c)}_t:=\tau_{c^{\alpha}t}$, then we derive from  (\ref{2561}) and (\ref{2562})
that
\begin{equation}\label{5555}
\{(c_iJ^{(i)}_{\tau_t^{(c)}}\exp(\xi^{(c,i)}_{\tau_t^{(c)}}))_{t\ge0},P_{y,\ln c+z}\}=
\{(c_iJ^{(i)}_{\tau_t}\exp(\xi^{(i)}_{\tau_t}))_{t\ge0},P_{y,z}\}\,.
\end{equation}
On the other hand, it is straightforward from the definitions that
\begin{equation}\label{3555}
X^{(i)}_{c^{\alpha}t}=c_iJ^{(i)}_{\tau_t^{(c)}}\exp(\xi^{(c,i)}_{\tau_t^{(c)}})\,.
\end{equation}
Then by taking $x=\varphi(y,z)$ so that $c\circ x=\varphi(y,\ln c+z)$ and $\p_x=P_{y,z}$, $\p_{c\circ x}=P_{y,\ln c+z}$, we 
derive from (\ref{5555}) and (\ref{3555}) that
\[\{(X_{c^{\alpha }t})_{t\ge0},\p_{c\circ x}\}=\{(c\circ X_t)_{t\ge0},\p_{x}\}\,.\]

Conversely, let $\{X,\p_x\}$ be an $E$-valued mssMp, with index $\alpha\in[0,\infty)^d$ and lifetime $\zeta$.  
Then by arguing exactly as in the direct part, we prove that the process $\{(J,\xi),P_{y,z}\}$ defined in the statement is an 
$S\times\mathbb{R}^d$-valued Hunt process, with lifetime 
$\zeta_*:=\int_0^{\zeta}\frac{ds}{|X_s^{(1)}|^{\alpha_1}\dots |X_s^{(d)}|^{\alpha_d}}$.\\

Now we have to check that the Hunt process $\{(J,\xi),P_{y,z}\}$ is a MAP. Let $(\mathcal{F}_t)_{t\ge0}$ be the filtration associated 
to   $X$ and completed with respect to the measures $(\p_x)_{x\in E}$. Define $A_t$ as in the statement if 
$t<\int_0^{\zeta}\frac{ds}{|X_s^{(1)}|^{\alpha_1}\dots |X_s^{(d)}|^{\alpha_d}}$, set $A_t=\infty$, if 
$t\ge\int_0^{\zeta}\frac{ds}{|X_s^{(1)}|^{\alpha_1}\dots |X_s^{(d)}|^{\alpha_d}}$ and note that for each $t$, $A_t$ is a stopping time of 
$(\mathcal{F}_t)_{t\ge0}$. Let us denote by $\theta_t$ the usual shift operator at time $t$ and note that for all $s,t\ge0$, 
\[A_{t+s}=A_t+\theta_{A_t}(A_s)\,.\]
Then let us prove that $\{(J,\xi),P_{y,z}\}$ is a MAP in the filtration $\mathcal{G}_t:=\mathcal{F}_{A_t}$. First observe that $(J,\xi)$ is 
clearly adapted to this filtration. Then from the strong Markov property of $\{X,\p_x\}$ applied at the stopping time $A_t$, we derive 
 from the definition of $\{(J,\xi),P_{y,z}\}$ that for any positive, Borel function $f$ and $x\in E$,
\begin{eqnarray*}
&&E_{\varphi^{-1}(x)}(f(J_{t+s},\xi_{t+s}-\xi_t),t+s<\zeta_*\,|\,\mathcal{G}_t)\\
&=&\e_{x}\left(f\left(\varphi^{-1}(X_{A_t+\theta_{A_t}(A_s)})-(0,\ln|X_{A_t}|)\right),A_t+\theta_{A_t}(A_s)<\zeta\,|\,\mathcal{F}_{A_t}\right)\\
&=&\e_{X_{A_t}}\left(f\left(\varphi^{-1}(X_{A_s})-(0,\ln z)\right),A_s<\zeta\right)_{z=|X_{A_{t}}|}\ind_{\{A_t<\zeta\}}\\
&=&\e_{\mbox{\footnotesize sign}(X_{A_t})}\left(f\left(\varphi^{-1}(X_{A_s})\right),A_s<\zeta\right)\ind_{\{A_t<\zeta\}}\\
&=&E_{J_t,0}(f(J_{s},\xi_{s}),s<\zeta_*)\ind_{\{t<\zeta_*\}}\,,
\end{eqnarray*}
where we have set $|x|=(|x_1|\,\dots,|x_d|)$, $\ln|x|=(\ln|x_1|\,\dots,\ln|x_d|)$ and where the third equality follows from the 
multi-self-similarity property of $\{X,\p_x\}$. We have obtained (\ref{3473}) and this ends the proof of the theorem.
\end{proof}

\noindent From this theorem, it is now easy to construct many examples of non trivial mssMp's. We can use for instance the 
MAP which is defined in (\ref{3492}) by coupling any continuous time Markov chain with an independent L\'evy process.   

\subsection{Asymptotic behavior of mssMp's.}\label{5732}

In this subsection we derive from Theorem \ref{2461} the behavior of a mssMp $\{X,\p_x\}$ as $t$ tends to $\zeta$. 
From this theorem and the construction of MAP's given in Subsection \ref{2004}, if the lifetime $\zeta_*$ of the underlying 
MAP $(J,\xi)$ under $P_{y,z}$ is finite with positive probability, then so is $\zeta$ and for $x=\varphi(y,z)$, the process $X$ under 
$\p_x$, jumps to 0 on the set $\zeta<\infty$, that is $X_{\zeta-}\neq0$, with positive probability. This situation has no interest for the 
problem we are studying and we will skip it. Moreover, we will always assume that $J$ is irreducible so that  if $E$
is composed of at least two orthants, then $\{X,\p_x\}$ does not pass through some of them a finite number of times. 
The reducible case can always be boiled down to the irreducible one from classical arguments. Therefore, we will assume that 
$\{X,\p_x\}$ is an $E$-valued mssMp absorbed at 0 whose underlying MAP $\{(J,\xi),P_{y,z}\}$ in the transformation given in 
Theorem \ref{2461} satisfies:

\begin{itemize}
\item[$(a)$] For all $y,z\in S\times\mathbb{R}^d$, $\zeta_*=\infty$, $P_{y,z}$-a.s.
\item[$(b)$] $J$ is irreducible.
\end{itemize}

\noindent Under assumption $(a)$ and from Theorem \ref{2461}, the lifetime of $\{X,\p_x\}$ is given by 
$\zeta=\int_0^\infty e^{\bar{\xi}_s}\,ds$. In order to determine conditions for this lifetime to be finite or infinite, we need reasonable 
assumptions on the asymptotic behavior of $\bar{\xi}$. These conditions will be ensured by Proposition \ref{2690}, so we will also 
assume: 

\begin{itemize}
\item[$(c)$] For each  $k=1,\dots,d$, there is $u\in\mathbb{R}\setminus\{0\}$ such that $A_k(u)$ is well defined. 
\end{itemize}

\noindent Let $\alpha\in[0,\infty)^d$ and note that the law of the process $(J,\bar{\xi})$ under $P_{y,z}$ only depends on $y$ and 
$\langle\alpha,z\rangle$. As already observed at the end of Subsection \ref{2004}, this process is actually an 
$S\times \mathbb{R}$-valued MAP under the family of probability measures defined by $\overline{P}_{y,t}:=P_{y,z}$, where $z$ 
is any vector such that $t=\langle\alpha,z\rangle$. The characteristic exponent of $\{(J,\bar{\xi}),\overline{P}_{y,t}\}$ is then given by 
\[\bar{A}(u)=A(u\cdot\alpha),\;\;\;u\in\mathbb{C}.\]
Since $A_k(u)=A(u\cdot e_k)$, under assumption $(c)$, there is $u\in\mathbb{R}\setminus\{0\}$ such that $\bar{A}(u)$ is well 
defined. Note that the Perron-Frobenius eigenvalue of $\bar{A}(u)$ depends on $\alpha$. Let us denote it by $\kappa_\alpha$. 
Applying Proposition \ref{2690}, we obtain  
\begin{equation}\label{7695}
\lim_{t\rightarrow\infty}\frac{\bar{\xi}_t}t=\kappa_\alpha,\;\;\;\mbox{$P_{y,z}-a.s.$, for all $y,z\in S\times \mathbb{R}^d$}.
\end{equation}       
Recall from Proposition XI.2.10 of \cite{LT8} that
$\lim_{t\rightarrow\infty}\bar{\xi}_t=-\infty$, $P_{y,z}$-a.s. for all $y,z\in S\times \mathbb{R}^d$ or 
$\lim_{t\rightarrow\infty}\bar{\xi}_t=+\infty$, $P_{y,z}$-a.s. for all $y,z\in S\times \mathbb{R}^d$ or 
$\liminf_{t\rightarrow\infty}\bar{\xi}_t=-\infty$ and $\limsup_{t\rightarrow\infty}\bar{\xi}_t=+\infty$, $P_{y,z}$-a.s. for all 
$y,z\in S\times \mathbb{R}^d$, according as $\kappa_\alpha<0,>0$ or $=0$. Then it follows from Theorem \ref{2461} that
$\zeta<\infty$, $\p_x-a.s.$ for all $x\in E$ or $\zeta=\infty$, $\p_x-a.s.$ for all $x\in E$ according as $\kappa_\alpha<0$ or 
$\kappa_\alpha>0$. The case where $\kappa_\alpha=0$ requires a bit more care and is proved in the following lemma which 
we have not found explicitly stated in the literature. 

\begin{lemma}\label{3596}
Assume that conditions $(a)$, $(b)$ and $(c)$ are satisfied. 
\begin{itemize}
\item[$(i)$] If $\kappa_\alpha<0$, then $\int_0^\infty e^{\bar{\xi}_s}\,ds<\infty$, $P_{y,z}$-a.s. for all $y,z\in S\times \mathbb{R}^d$,
\item[$(ii)$] If $\kappa_\alpha\ge0$, then $\int_0^\infty e^{\bar{\xi}_s}\,ds=\infty$, $P_{y,z}$-a.s. for all $y,z\in S\times \mathbb{R}^d$.
\end{itemize}
\end{lemma}
\begin{proof} The cases where $\kappa_\alpha<0$ and $\kappa_\alpha>0$ follow directly from (\ref{7695}). Let us assume now that 
$\kappa_\alpha=0$. Then as recalled above, $\liminf_{t\rightarrow\infty}\bar{\xi}_t=-\infty$ and 
$\limsup_{t\rightarrow\infty}\bar{\xi}_t=+\infty$, $P_{y,z}$-a.s. for all $y,z\in S\times \mathbb{R}^d$. 
Fix $a>0$ and let  $\tau_{a}^{(n)}=\inf\{t\ge\sigma_{a}^{(n-1)}:\bar{\xi}_t>a\}$ and $\sigma_{a}^{(n)}=
\inf\{t\ge\tau_{a}^{(n)}:\bar{\xi}_t<a\}$, $n\ge0$, with  $\sigma_{a}^{(0)}=0$. Then $\tau_{a}^{(n)}$ is a sequence of stopping times in 
the filtration $(\mathcal{G}_t)_{t\ge0}$ of $\{(J,\xi),P_{y,z}\}$ such that $\lim_{n\rightarrow\infty}\tau_a^{(n)}=+\infty$, $P_{y,z}$-a.s. for 
all $y,z\in S\times \mathbb{R}^d$.  Moreover,
\begin{equation}\label{5611}
\int_0^\infty e^{\bar{\xi}_s}\,ds\ge\sum_{n=1}^\infty\int_{\tau_a^{(n)}}^{\sigma_a^{(n)}}e^{\bar{\xi}_s}\,ds\ge 
\exp(a)\sum_{n=1}^\infty(\sigma_a^{(n)}-\tau_a^{(n)})\,.
\end{equation}
Recall that $J$ is irreducible and let $\pi$ be its invariant measure on $S$. We derive from the definition of MAP's that the sequence 
$\sigma_a^{(n)}-\tau_a^{(n)}$, $n\ge0$ is stationary under $P_{\pi,0}$ and since $P_{\pi,0}(\sigma_a^{(1)}-\tau_a^{(1)}>0)=1$,
we obtain that $P_{\pi,0}(\sum_{n=1}(\sigma_a^{(n)}-\tau_a^{(n)})=\infty)=1$. On the other hand, from the Markov property, for all 
$y\in S$ and $n\ge1$,
\[P_{y,0}\left(\sum_{k=1}^\infty(\sigma_a^{(k)}-\tau_a^{(k)})=\infty\right)=
E_{y,0}\left(P_{J_{\tau_a^{(n)}},0}\left(\sum_{k=1}^\infty(\sigma_a^{(k)}-\tau_a^{(k)})=\infty\right)\right).\] 
Since, the finite valued Markov chain $(J_{\tau_a^{(n)}})_{n\ge1}$ converges in law to $\pi$, by letting $n$ go to $\infty$ in this equality, 
we obtain that for all $y\in S$, 
$P_{y,0}\left(\sum_{k=1}^\infty(\sigma_a^{(k)}-\tau_a^{(k)})=\infty\right)=
P_{\pi,0}(\sum_{n=1}^\infty(\sigma_a^{(n)}-\tau_a^{(n)})=\infty)=1$, so that from inequality (\ref{5611}), 
$\int_0^\infty e^{\bar{\xi}_s}\,ds=\infty$, $P_{y,0}$-a.s. for all $y\in S$. But the definition of MAP's clearly implies that this holds 
$P_{y,z}$-a.s.~for all $y,z\in S\times \mathbb{R}^d$.
\end{proof}
\noindent Note that the trichotomy of Lemma \ref{3596} holds for any MAP with values in $F\times \mathbb{R}$, where $F$ is any 
finite set.\\

Recall from Subsection \ref{3656} the notation 
\[\lim_{t\uparrow\zeta} X_t=X_{\zeta-},\]
when this limit exists and where $\zeta$ is supposed to be finite or infinite. We remind that this limit is to be understood in the 
topology of the compact space $E_0$ and that, as already observed before Proposition \ref{3673}, when $\zeta<\infty$, the 
existence of $X_{\zeta-}$ is guaranteed by the fact that $\{X,\p_x\}$ is a Hunt process. Recall also the definition of $\chi'_k(0)$ 
from Subsection \ref{2004}.\\

\begin{theorem}\label{2774}
Assume that conditions $(a)$, $(b)$ and $(c)$ are satisfied. 
\begin{itemize}
\item[$(i)$] If $\kappa_\alpha<0$, then $\p_x(\zeta<\infty)=1$, for all $x\in E$ and if $\kappa_\alpha\ge0$, then 
$\p_x(\zeta=\infty)=1$, for all $x\in E$. 
\item[$(ii)$]  If one of the following conditions is satisfied: 
\begin{itemize}
\item[$(a)$]  $\chi_k'(0)<0$ or $\chi'_k(0)>0$, for some $k=1,\dots,d$,
\item[$(b)$]  $\chi_k'(0)=0$, for all $k=1,\dots,d$ and $\kappa_\alpha<0$,
\item[$(c)$] $\chi_k'(0)=0$, for all $k=1,\dots,d$ and $\kappa_\alpha>0$, 
\end{itemize}
then $\p_x(X_{\zeta-}=0)=1$, for all $x\in E$.  
\item[$(iii)$] For all $x\in E$, $\p_x(X_{\zeta-}\;\mbox{exists and}\;X_{\zeta-}\in \mathbb{R}^d\setminus\{0\})=0$. 
\end{itemize}
\end{theorem}
\begin{proof} First is it useful to observe that under our assumptions, since $\zeta_*=\infty$, $P_{y,z}$-a.s., for all 
$y,z\in S\times\mathbb{R}^d$,
\begin{equation}\label{8867}
\lim_{t\uparrow\zeta}\tau_t=\infty,\;\;\;\mbox{$\p_{y,z}$-a.s., for all $y,z\in S\times\mathbb{R}^d$.}
\end{equation}
The first assertion is an immediate consequence of the expression $\zeta=\int_0^\infty e^{\bar{\xi}_s}\,ds$ in Theorem \ref{2461} 
and Lemma \ref{3596}. The second one follows from the expression $X_t=\varphi(J_{\tau_t},\xi_{\tau_t})$. Indeed, if $\chi_k'(0)<0$ 
or $\chi'_k(0)>0$, for some $k=1,\dots,d$, then from Proposition XI.2.10 of \cite{LT8}, as $t$ tends to $\zeta$, the coordinate 
$X^{(k)}_t=J^{(k)}_{\tau_t}e^{\xi^{(k)}_{\tau_t}}$ tends to 0 if $\chi_k'(0)<0$ and to $\infty$ if $\chi_k'(0)>0$, $P_{y,z}$-a.s. for all
$(y,z)\in S\times\mathbb{R}^d$. Therefore $X_t$ tends to 0 as $t$ tends to $\zeta$, in the topology of the compact set $E_0$, 
$\p_x$-a.s. for all $x\in E$. If $\chi_k'(0)=0$, for all $k=1,\dots,d$ and $\kappa_\alpha<0$ (resp. $\kappa_\alpha>0$), then from 
Proposition XI.2.10 of \cite{LT8}, the process $|(X_t^{(1)})^{\alpha_1}\dots(X_t^{(d)})^{\alpha_d}|=e^{\bar{\xi}_{\tau_t}}$ tends to 
0 (resp. to $\infty$) as $t$ tends to $\infty$. Hence $X_t$ tends to 0 in the topology of $E_0$.\\

Let us now prove $(iii)$. Assume that $x=\varphi(y,z)$ is such that $X_{\zeta-}$ exists $\p_x$-a.s. then for any $k=1,\dots,d$,
 from  (\ref{8867}) and Proposition XI.2.10 of \cite{LT8}, the process $\xi^{(k)}_{\tau_t}$ either tends to $-\infty$ or to $+\infty$ or
oscillates $\p_x$-a.s.  as $t$ tends to $\zeta$. Therefore the limit 
$\lim_{t\uparrow\zeta}X^{(k)}_t=\lim_{t\uparrow\zeta}J^{(k)}_{\tau_t}e^{\xi^{(k)}_{\tau_t}}$ cannot belong to 
$\mathbb{R}\setminus\{0\}$.
\end{proof}

\noindent Note that parts  $(i)$ and $(ii)$ of Theorem \ref{2774} extend Proposition \ref{3673} to the case where $E$ 
is any state space, but with additional assumptions.  It seems that it is not possible to conclude in the case where $\chi_k'(0)=0$, 
for all $k=1,\dots,d$ and $\kappa_\alpha=0$. We are only able to construct examples such that for all $x\in E$, $X_t$ has no 
limit $\p_x$-a.s., when $t$ tends $\infty$. Note also that part $(iii)$ of Theorem \ref{2774} completes part 1.~of Proposition 
\ref{2683} where it was proved that the set $\{x\in\mathbb{R}^d\setminus\{0\}:x_1x_2\dots x_d=0\}$ is absorbing. Actually under 
our assumptions this set is a.s.~never attained.\\  

\begin{remark}   
It is important to note that $\{X,\p_x\}$ has finite or infinite lifetime depending on the valued of $\alpha$. Changing 
$\alpha$ may change a finite lifetime to an infinite one. This makes another difference with self-similar Markov processes, 
where the index can be changed simply by raising the process to some power. Then the lifetime remains unchanged.  
However, recall that $\{X,\p_x\}$ is a self-similar Markov process with index $\alpha_1+\dots+\alpha_d$. Therefore the finiteness 
of the lifetime of $\{X,\p_x\}$ does not depend on the sum $\alpha_1+\dots+\alpha_d$.
\end{remark}

\newpage

\end{document}